\documentclass[10pt]{article}

\let\oldmarginpar\marginpar
\renewcommand\marginpar[1]{\-\oldmarginpar[\raggedleft\footnotesize #1]%
{\raggedright\footnotesize #1}}

\usepackage{lineno}
\usepackage{mathptmx}
\usepackage{amsmath}
\usepackage{amscd}
\usepackage{amssymb}
\usepackage{amsthm}
\usepackage{xspace}
\usepackage[all,tips]{xy}
\usepackage[dvips]{graphicx}
\usepackage{verbatim}
\usepackage{syntonly}
\usepackage{hyperref}
\usepackage{arydshln}
\usepackage{lineno,color}
\usepackage{relsize}

\theoremstyle{plain}
\newtheorem{thm}{Theorem}[section]
\newtheorem{cor}[thm]{Corollary}
\newtheorem{prop}[thm]{Proposition}
\newtheorem{lemma}[thm]{Lemma}
\newtheorem{ques}{Question}

\newtheorem{innercustomthm}{Theorem}
\newenvironment{customthm}[1]{\renewcommand \theinnercustomthm{#1}\innercustomthm}{\endinnercustomthm}

\theoremstyle{definition}

\newtheorem{defn}[thm]{Definition}

\DeclareMathOperator{\SL}{SL} 
\DeclareMathOperator{\GL}{GL} 
 
 \DeclareMathOperator{\SO}{SO}

\DeclareMathOperator{\ad}{ad}

\DeclareMathOperator{\Log}{Log}

\DeclareMathOperator{\Nilp}{\textbf{N}}

\DeclareMathOperator{\rank}{rank}

\newcommand{\eps}{\varepsilon}

\newcommand{\bdef}{\overset{\text{def}}{=}}

\newcommand{\al}{\alpha}

\newcommand{\ga}{\gamma}
\newcommand{\Ga}{\Gamma}

\newcommand{\La}{\Lambda}

\newcommand{\innp}[1]{\left< #1 \right>}
\newcommand{\abs}[1]{\left\vert#1\right\vert}
\newcommand{\set}[1]{\left\{#1\right\}}

\newcommand{\pr}[1]{\left( #1 \right) }

\newcommand{\B}[1]{\ensuremath{\mathbf{#1}}}

\newcommand{\Fr}[1]{\ensuremath{\mathfrak{#1}}}

\newcommand{\Q}{\ensuremath{\mathbb{Q}}}
\newcommand{\R}{\ensuremath{\mathbb{R}}}
\newcommand{\Z}{\ensuremath{\mathbb{Z}}}
\newcommand{\C}{\ensuremath{\mathbb{C}}}
\newcommand{\map}[3]{#1\colon #2 \rightarrow #3}

\usepackage{fancyhdr}

\begin{document}

	
\title{\textbf{Coarse models of homogeneous spaces and translations-like actions}}
\author{D. B. McReynolds, Mark Pengitore} 
\maketitle	


\begin{abstract}
\noindent For finitely generated groups $G$ and $H$ equipped with word metrics, a translation-like action of $H$ on $G$ is a free action where each element of $H$ moves elements of $G$ a bounded distance. Translation-like actions provide a geometric generalization of subgroup containment. Extending work of Cohen, we show that cocompact lattices in a general semisimple Lie group $\mathbf{G}$ that is not isogenous to $\mathrm{SL}(2,\mathbb{R})$ admit translation-like actions by $\Z^2$. This result follows from a more general result. Namely, we prove that any cocompact lattice in the unipotent radical $\mathbf{N}$ of the Borel subgroup $\mathbf{AN}$ of $\mathbf{G}$ acts translation-like on any cocompact lattice in $\mathbf{G}$. We also prove that for noncompact simple Lie groups $G,H$ with $H<G$ and lattices $\Gamma < G$ and $\Delta < H$, that $\Gamma/\Delta$ is quasi-isometric to $G/H$ where $\Gamma/\Delta$ is the quotient via a translation-like action of $\Delta$ on $\Gamma$. 
\end{abstract}

\section{Introduction}

Given a Lie group $\B{G}$, every cocompact lattice $\Gamma < \B{G}$ with a finite word metric is a coarse geometric model of $\B{G}$ (e.g.~the inclusion map is a quasi-isometry). One theme in the study of lattices is how much of the structure of $\B{G}$ is captured in the structures on the lattices $\Gamma$. When $\B{G}$ is a non-compact real simple Lie group of real rank at least two, Margulis established that these lattices are arithmetic which is one of the strongest ways that $\Gamma$ can capture the structure of $\B{G}$. He also directly related the finite dimensional representation theory of $\Gamma$ with that of $\B{G}$ via super-rigidity. These lattices are also conjectured by Serre to have the congruence subgroup property, which shows that the finite representation theory of $\Gamma$ functions through the structure of $\B{G}$. 

Our interest in this article is in subgroups. Given a Lie group $\B{G}$ and closed subgroup $\B{H} \leq \B{G}$, two structures are the homogenous space $\B{G}/\B{H}$ and the associated foliation of $\B{G}$ via distinct $\B{H}$--cosets $g\B{H}$. Given a cocompact lattice $\Gamma \leq \B{G}$, we define $\Delta = \B{H} \cap \Gamma$ and ask if $\Gamma/\Delta$ is a coarse model for $\B{G}/\B{H}$? When $\Delta \leq \B{H}$ is a cocompact lattice, $\Gamma/\Delta$ is a coarse model for $\B{G}/\B{H}$. Likewise, the coset foliation on $\Gamma$ via $\gamma \Delta$ is a coarse model for the $\B{H}$--coset foliation. Unfortunately, the intersection $\Delta = \B{H} \cap \Gamma$ can vary (depending on $\Gamma$ and $\B{H}$) from trivial to a cocompact lattice in $\B{H}$, and is typically trivial. For instance, there are infinitely many commensurability classes of arithmetic lattices $\Gamma < \SL(2,\C)$ such that $\B{H} \cap \Gamma$ is a cocompact lattice for countably many $\B{H}$ that are conjugate to $\SL(2,\R)$. However, there also exist infinitely many commensurability classes of arithmetic lattices $\Gamma < \SL(2,\C)$ such that $\B{H} \cap \Gamma$ is never a lattice for any $\B{H}$ that is conjugate to $\SL(2,\R)$. Regardless, all of these lattices have quasi-isometric surface subgroups $\Delta$ by Kahn--Markovic \cite{KM}. For these subgroups $\Delta$, the space $\Gamma/\Delta$ gives a coarse model for $\B{G}/\B{H}$ despite $\Delta$ not being a subgroup of $\B{H}$. We take an alternative approach to finding models for $\B{G}/\B{H}$.

Given a group $H$ and a metric space $(X,d)$ with a free (left) $H$--action, we say that $H$ acts \textbf{translation-like} on $X$ if $\sup\set{d(x,h \cdot x)~:~x \in X} < \infty$ for each $h \in H$; an action satisfying this condition is called \textbf{wobbling}. Our present interest is when $X = G$ is a finitely generated group equipped with a word metric associated to a finite generating subset. Whyte introduced translation-like actions as a geometric coarsification of subgroups. Indeed, when $H \leq G$, the right action of $H$ on $G$ is free and translation-like for any finite generating subset of $G$. In an effort to justify this view, Whyte established a coarse geometric result in relation to the von Neumann--Day conjecture. The conjecture asserts that a group $G$ is non-amenable if and only if $G$ contains a non-abelian free subgroup, which by Ol'shanskii \cite{olshanskii} is known to be false. On the other hand, Whyte \cite{Whyte1} proved a coarsification of this conjecture, establishing that $G$ is non-amenable if and only if $G$ admits a translation-like action by a non-abelian free group.

In 1902, Burnside asked if every infinite, finitely generated group $G$ contains an element of infinite order, and Golod--Shafarevich \cite{Golod_Safarevic} answered Burnside's question in the negative by providing examples of finitely generated infinite torsion groups. Seward \cite{Seward1} took a similar approach as Whyte to Burnside's problem, proving that a finitely generated group $G$ is infinite if and only if $G$ admits a translation-like action by $\Z$.

With translation-like actions that are sufficiently well behaved, we provide a method to construct a model for the homogeneous space $\B{G} / \B{H}$ that is compatible with models for the Lie groups $\B{G}$ and $\B{H}$ given by cocompact lattices $\Delta < \B{H}$ and $\Gamma < \B{G}$. Suppose that $\Delta$ admits a translation-like action on $\Gamma$ where the orbits of the action of $\Delta$ on $\Gamma$ are coarsely embedded and are contained in cosets of $\B{H}$ in $\B{G}$. Moreover, suppose that the quotient of $\Gamma$ by the translation-like action of $\Delta$ admits a natural metric with a natural inclusion into $\B{G} / \B{H}$ that is coarsely dense. We then say that the translation-like action of $\Delta$ on $\Gamma$ gives rise to a coarse model of $\B{G} / \B{H}$ and denote it as $\Ga / \Delta$.

Following Seward and Whyte, Cohen \cite{Cohen1} investigated the geometric coarsification of a question due to Gersten--Gromov (see \cite[Ques 1.1]{Best}). Specifically, if $G$ admits a finite $K(G,1)$ and contains no Baumslag--Solitar subgroups $BS(m,n)$, then is $G$ hyperbolic? Like the von Neumann--Day conjecture and Burnside's question, this question is known to have a negative answer, and in fact, there are many counterexamples to the Gersten--Gromov question. For example, Rips \cite{rips} proved that there exists a $C^\prime(1/6)$ small cancellation group with a finitely generated but not finitely presentable subgroup $H$. Since $C^\prime(1/6)$ small cancellation groups are hyperbolic, the subgroup $H$ cannot contain any Baumslag--Solitar subgroups which gives a counterexample to the Gersten conjecture. Even if we restrict ourselves to the class of finitely presentable groups, we have counterexamples. Brady \cite{Brady} using branched coverings of cubical complexes to produce a hyperbolic group with a finitely presented subgroup that is not hyperbolic which provides finitely presentable counterexample to the Gersten conjecture.

The geometric coarsification of the Gersten--Gromov question is that a group $G$ with a finite $K(G,1)$ is hyperbolic if and only if $G$ does not admit a translation-like action by any Baumslag--Solitar group. The main result of \cite{Cohen1} proved that cocompact lattices in $\SO(3,1)$ admit translation-like actions by $\Z^2$, proving that the geometric coarsification of the Gersten--Gromov question is false. Moreover, by inspecting the construction in \cite{Cohen1}, we see that the translation-like action of $\Z^2$ on cocompact lattices in $\SO(3,1)$ gives rise to a coarse model for $\SO(3,1) / \R^2$ which can be seen as the space of horospheres of $3$-dimensional hyperbolic space.

Our first result extends \cite{Cohen1} to all cocompact lattices in all semisimple Lie groups. Fixing an Iwasawa decomposition of $\B{G} = \B{K}\B{A}\B{N}$, when $\Gamma < \B{G}$ is a non-cocompact lattice, then $\Delta = \Gamma \cap \B{N}$ is a cocompact lattice in $\B{N}$. The Lie group $\B{N}$ is a connected, simply connected nilpotent Lie group and so $\Delta < \B{N}$ is a torsion-free, finitely generated nilpotent group. When $\Gamma < \B{G}$ is a cocompact lattice, then $\Gamma \cap \B{N}$ is trivial. Despite it being impossible for $\Gamma$ to have torsion-free nilpotent subgroups besides $\Z$, the lattices $\Gamma$ do admit translation-like actions by the lattices in $\B{N}$ that give rise to coarse models for $\B{G} / \B{N}$.

\begin{thm}\label{main_result}
Let $\B{G}$ be a semisimple Lie group with an Iwasawa decomposition $\B{G} = \B{K}\B{A}\B{N}$. If $\Gamma < \B{G}$ and $\Delta < \B{N}$ are cocompact lattices, then $\Gamma$ admits a translation-like action by $\Delta$. Moreover, we can choose this translation-like action to give rise to a coarse mode $\Ga / \La$ of the homogeneous space $\textbf{G} / \textbf{N}$. Finally, given distinct lattices $\Ga_1, \Ga_2 < \B{G}$ and $\Delta_1, \Delta_2 < \B{N}$, we have the coarse models $\Ga_1/\Delta_1$ and $\Ga_2 / \Delta_2$ for $\B{G} / \B{N}$ are bi-Lipschitz.
\end{thm}

One immediate corollary of this theorem is the following.

\begin{cor}\label{main_cor}
Let $\B{G}$ be a semisimple Lie group which is not isogenous to $\SL(2,\R)$. If $\Gamma < \B{G}$ is cocompact lattice, then $\Gamma$ admits a translation-like action by $\Z^2$.
\end{cor}

This corollary generalizes the main result of \cite{Cohen1}. More recently, Jiang \cite{jiang} proved that the lamplighter group admits no translation-like actions by Baumslag--Solitar groups. As the lamplighter group is not finitely presentable, it cannot be hyperbolic. Hence, this provides a counterexample for the other direction of the geometric coarsification of the Gersten--Gromov question. In particular, there are hyperbolic groups that admit actions by Baumslag--Solitar groups and there exist non-hyperbolic groups which do not admit any translation-like actions by a Baumslag--Solitar group. 

\begin{ques}
Does there exist a non-hyperbolic, finitely presentable group that does not admit a translation-like action by any Baumslag--Solitar group?
\end{ques}

We give an outline of the proof of our first theorem which follows the proof of the main theorem of \cite{Cohen1}.  Using unipotent flows, we construct a net in $\B{G}/\B{K}$ which is bi-Lipschitz to our group $\Gamma$ on which $\Delta$ admits a translation-like action. The unipotent subgroups of the Iwasawa decomposition with the induced metric are bi-Lipschitz to $\B{N}$ with a left invariant metric in which $\Delta$ is a cocompact lattice. The nilpotent Lie groups $\B{N}$ admit natural scaling automorphisms which we use to shrink or expand the copy of $\Delta$ in each coset $(a,\textbf{N})$ where $a \in \mathbb{Z}^{\rank(\textbf{G})}$ as $a$ varies to account for the changes in the induced geometry of each translate of the unipotent subgroup. Since each layer of this net is a copy of $\Delta$, we act on these layers by right translation. The actions on the layers combine together to give an action on the entire net that is translation-like. Through the bi-Lipschitz equivalence of $\Gamma$ with this net, we obtain a translation-like action of the group $\Delta$ on $\Gamma$.

The last theorem of our note constructs coarse models for homogeneous spaces of the form $\B{G} / \B{H}$ where both $\B{G}$ and $\B{H}$ are noncompact real simple Lie groups using cocompact lattices in $\B{G}$ and $\B{H}$. We refer the reader to Definition \ref{D:CorMod} for the definition of a coarse model.

\begin{thm}\label{second_result}
Let $\B{G}$ and $\B{H}$ be $\Q$--defined noncompact real simple Lie groups such that $\B{H} \leq \B{G}$. If $\Delta < \B{H}$ and $\Gamma < \B{G}$ are cocompact lattices, then $\Delta$ admits a translation-like action on $\Gamma$. Moreover, we can choose this translation-like such that $\Ga / \Delta$ is a coarse model for $\B{G} / \B{H}$. Finally, given distinct lattices $\Ga_1, \Ga_2 < \B{G}$ and $\Delta_1, \Delta_2 < \B{H}$, the spaces $\Ga_1/\Delta_1$ and $\Ga_2 / \Delta_2$ for $\B{G} / \B{H}$ are bi-Lipschitz.
\end{thm}

The proof of this theorem follows from basic structural results of simple Lie groups.

\section{Background}

For a group $G$ and $g, h \in G$, the commutator is denoted by $g$ and $h$ as $[g,h] = g^{-1}h^{-1}gh$. For subgroups $A,B \leq G$, the subgroup generated by $\set{[a,b]~:~a\in A,~b \in B}$ is denoted by $[A,B]$. The $i$-th step of the lower central series of $G$ is denoted as $G_i$. When $N$ is a nilpotent group, we denote its step length as $c(N)$.
	
\subsection{Lie groups and Lie algebras}

Lie groups will be typically denoted by $\textbf{G}$ with Lie algebras given by $\Fr{g}$. The Lie bracket of $X$ and $Y$ will be denoted by $[X,Y]$. Inner products will be denoted $\innp{ \cdot, \cdot }$. Left translation by a group element $g \in \textbf{G}$ will be denoted by $L_g$. The $i$-th step of the lower central series of a Lie algebra $\Fr{g}$ will be denoted by $\Fr{g}_i$. The tangent space of $\textbf{G}$ at any element $g \in \textbf{G}$ will be denoted by $T_g(\textbf{G})$.
	
Given a connected Lie group $\textbf{G}$ with Lie algebra $\Fr{g}$ and $g \in \textbf{G}$, the map $\map{L_g}{\textbf{G}}{\textbf{G}}$ given by $L_g(x) = g \cdot x$ is a diffeomorphism of $\textbf{G}$ for all $g \in \textbf{G}$. Thus,  the tangent space $T_g(\textbf{G})$ can be identified as $(dL_g)_1(T_1(\textbf{G}))$ where $(dL_g)_1$ is the linear isomorphism from $T_1(\textbf{G})$ to $T_g(\textbf{G})$. Fixing a positive definite bilinear form $\innp{\cdot,\cdot}$ on $\mathfrak{g} = T_1(\B{G})$, we have a left invariant Riemannian metric on $\B{G}$ defined via 
\[ \innp{X,Y}_g = \innp{dL_{g^{-1}}(X),dL_{g^{-1}}(Y)} \] 
for all $X,Y \in T_g(\B{G})$ and for all $g \in \textbf{G}$. For $X \in \Fr{g}$, we have the linear endomorphism $\ad_X\colon \mathfrak{g} \to \mathfrak{g}$ given by $\ad_X(Y) = [X,Y]$.

Given a group $G$, we define the \textbf{lower central series of $G$} recursively by $G_1 = G$ and $G_i = [G,G_{i-1}]$ for $i>1$. We say that $G$ is \textbf{nilpotent of step size $c$} if $c$ is the minimal natural number such that $G_{c+1} = \set{1}$. If the step size is unspecified, we just say that $G$ is a nilpotent group. The \textbf{lower central series for a Lie algebra $\Fr{g}$} is defined recursively by $\Fr{g}_1 = \Fr{g}$ and $\Fr{g}_i = [\Fr{g},\Fr{g}_{i-1}]$ for $i>1$. We say that $\Fr{n}$ is \textbf{nilpotent of step length $c$} if $c$ is the minimal natural number such that $\Fr{n}_{c+1} = \set{0}$. If the step size is unspecified, we just say that $\Fr{n}$ is a nilpotent Lie algebra.

Given a Lie group $\B{G}$ and a left Haar measure $\mu$, we say that a discrete subgroup $\Gamma<G$ is a \textbf{lattice} if $\mu(\Gamma \backslash \B{G}) < \infty$. When $\Gamma \backslash \B{G}$ is compact, we say $\Gamma$ is \textbf{cocompact}. If $\B{G} < \GL(n,\C)$ is a $\Q$-defined linear group, the group of integral points is defined by $\B{G}(\Z) = \B{G} \cap \GL(n,\Z)$. 

\subsection{Coarse Geometry and UDBG spaces}

Given metric spaces $(X_1,d_1)$ and $(X_2,d_2)$, we say $X_1$ and $X_2$ are \textbf{quasi-isometric} if there exists a function $f\colon (X_1,d_1)\to (X_2,d_2)$ and constants $A \geq 1$, $B \geq 0$, and $C \geq 0$ such that
\[ \frac{1}{A}d_1(x,y) - B \leq d_2(f(x),f(y)) \leq A d_1(x,y) + B, \] 
for all $x, y \in X_1$, and for each $z \in X_2$, there exists an element $x \in X_1$ such that $d_2(z,f(x)) \leq C$. We call the map $f$ a \textbf{quasi-isometry} between $(X_1,d_1)$ and $(X_2,d_2)$. If the above map is bijective and if $B= 0$, we call the map $f$ a \textbf{bi-Lipschitz map} and say that the metric spaces $\pr{X_1,d_1}$ and $\pr{X_2,d_2}$ are \textbf{bi-Lipschitz}.

We introduce some conditions on discrete metric spaces that induce some regularity. We say a metric space $(X,d)$ is \textbf{uniformly discrete} if 
\[ \inf\set{d(x_1,x_2)~:~ x_1,x_2 \in X \text{ and } x_1 \neq x_2}>0. \]
A discrete metric space $(X,d)$ has \textbf{bounded geometry} if for all $r>0$, there exists a constant $C_r > 0$ such that $|B_r(x)| \leq C_r$ for all $x \in X$.
We call a uniformly discrete metric space of bounded geometry a \textbf{UDBG space}. 

We are interested in a particular class of UDBG spaces seen in the following definition
\begin{defn}
Let $X$ be a UDBG space. If $F \subset X$ and $r \in \mathbb{N}$, then the \textbf{$r$--boundary of $F$ in $X$} is given by
\[ \partial_r^X (F) \bdef \set{x \in X - F ~:~ \text{ there exists } y \in Y \text{ such that } d(x,y) \leq r }. \]
A \textbf{F{\o}lner sequence for $X$} is a sequence $\{F_{i}\}_{i \in \mathbb{N}}$ of non-empty finite subsets of $X$ such that for all $r \in N$, we have
\[ \lim_{n \rightarrow \infty} \frac{\abs{\partial_r^X (F_n)}}{\abs{F_n}} = 0. \]
We say that a UDBG space is \textbf{non-amenable} if it admits no F{\o}lner sequences.
\end{defn}

The following property of non-amenable UDBG spaces is of particular importance to us.

\begin{prop}\label{udbg_nonamenable}
Let $(X_1,d_1)$ and $(X_2,d_2)$ be non-amenable UDBG spaces, and suppose that $f\colon X_1 \to X_2$ is a quasi-isometry. Then $f$ is bounded distance from a bi-Lipschitz map $F\colon X_1\to X_2$.
\end{prop}

\begin{proof}
Since $X_1$ and $X_2$ are non-amenable, we have that $H_0^{uf}(X_1) = 0$ and $H_0^{uf}(X_2) = 0$ by \cite[Thm 3.1]{block_weinberger} where $H_0^{uf}(X_1)$ and $H_0^{uf}(X_2)$ denote the $0$-th uniformly finite homology groups of $X_1$ and $X_2$. Denoting $[X_1]$ and $[X_2]$ as the characteristic classes of $X_1$ and $X_2$, we have that  $[X_1] = 0$ and $[X_2] = 0$. Thus, if $f_*\colon H_0^{uf}(X_1) \to H_0^{uf}(X_2)$ is the map of $0$-th uniformly finite homology induced by the quasi-isometry $f$, we have $f_*([X_1]) = [X_2]$. Hence, \cite[Thm 1.1]{Whyte1} implies that $f$ is bounded distance from a bi-Lipschitz map.
\end{proof}
	
We finish this section by noting some straightforward properties of translation-like actions. In particular, translation-like actions respect bi-Lipschitz equivalences of metric spaces and satisfy transitivity properties as seen in the following lemmas. As these lemmas are straightforward, we omit the proofs for brevity.
	
\begin{lemma}\label{bilipschitz_equivalent}
Let $G$ be a finitely generated group that acts translation-like on $(X_1,d_1)$, and suppose that $(X_1,d_1)$ is bi-Lipschitz to $(X_2,d_2)$ via the map $F$. Then $G$ admits a translation-like action on $(X_2,d_2)$ via the action $g \cdot x = F(g \cdot F^{-1}(x))$.
\end{lemma}

\begin{lemma}\label{bilipschitz_acting_groups}
Let $H,G$ be finitely generated groups equipped with word metrics, and let $\pr{X,d}$ be a metric space. Suppose that $H$ that is bi-Lipschitz to $G$ via the map $F$ and that $G$ acts translation-like on $(X,d)$. If $\La$ is a set of orbit representatives of the action of $G$ on $X$, then $H$ acts translation-like on $(X,d)$ via $h \cdot (x \cdot g) = x \cdot F(F^{-1}(g) \cdot h)$ for $x \in \La$ where we write the action on the right.
\end{lemma}

\begin{lemma}\label{transitivity_translation_like}
Let $H,G$ be finitely generated groups equipped with word metrics, and let $\pr{X,d}$ be a metric space. Suppose that $H$ acts translation-like on $G$ and that $G$ acts translation-like on $(X,d)$. Then $H$ acts translation-like on $(X,d)$.
\end{lemma}

\subsection{Coarse models for homogeneous spaces}

We start this subsection with the following definition.

\begin{defn}
Let $X$ be a metric space and suppose that $G$ is a finitely generated group that admits at translation-like action on $X$. A \textbf{chain} between $x$ and $y$ in $X$ is a sequence of points $\{x_i, y_i\}_{i=1}^k$ such that $x = x_1$, $y = y_k$, and for each $1 \leq i \leq k -1$, there exists a $g_i \in G$ such that $g_i \cdot y_i = x_{i+1}$.
\end{defn}

With the notion of chains between points in a metric space being acted on translation-like, we can define a natural quotient of metric space by the translation-like action by some finitely generated group.

\begin{defn}
Let $(X,d)$ be a metric space, and suppose that $G$ is a finitely generated group that admits a translation-like action on $X$. We define a distance function $d\colon X \times X \to \R_{\geq 0}$ on the quotient $X / \sim$ by
\[ d_{X/G}([x],[y]) = \inf \set{ \sum_{i=1}^k d(x_i,y_i) ~:~ \set{x_i,y_i}_{i=1}^k \text{ is a chain from } x \text{ to } y}. \]
The space $X/\sim$ endowed with the function $d_{X/G}(\cdot, \cdot)$ is call the \textbf{translation-like geometric quotient of $X$ by $G$}.
\end{defn}

For a general metric space $(X,d)$ which admits a translation-like action by a group $G$, we have that $X/ G$ is not necessarily a metric space. However, when $X$ is a UDBG space, the $X/G$ is a metric space as seen in the following proposition.

\begin{prop}
Let $X$ be a UDBG space, and suppose that $G$ admits a translation-like action on $X$. Then $X/ G$ is a metric space.
\end{prop}

\begin{proof}
To begin, $d_{X / G}([x], [y]) = d_{X / G}([y], [x])$ is clear. As $X$ is a UDBG space, we have that 
\[ \inf\set{d(x,y)~:~ x,y \in X, x \neq y } > 0. \] 
In particular, if $[x]$, $[y]$ are distinct equivalence classes in $X/ G$, then $d_{X / G}([x], [y]) > 0$. For the triangle inequality, let $\{p_i, q_i \}_{i=1}^{k}$ be a chain from $x$ to $y$, and let $\{p_t^\prime, q_t^\prime\}_{t=1}^s$ be a chain from $y$ to $z$. We then have that $\set{p_i, q_i}_{i=1}^k \cup \set{p_t^\prime, q_t^\prime}_{t=1}^s$ is a chain from $x$ to $z$. We may write
\[ d_{X / G}([x],[y]) \leq \sum_{i=1}^k d(p_i, p_i) + \sum_{t=1}^s d(p_t^\prime, q_t^\prime). \]
By definition, we note that
\begin{eqnarray*}
d_{X/G}([x],[y]) + d_{X/G}([y],[z]) &=&  \inf \set{ \sum_{i=1}^k d(p_i,q_i)~:~ \{p_{i},q_i\}_{i=1}^k \text{ is a chain from $x$ to $y$} }\\ &+&  \inf \set{ \sum_{i=1}^s d(p_t^\prime,q_t^\prime) ~:~ \{p_{t}^\prime,q_t^\prime\}_{t=1}^k \text{ is a chain from $y$ to $z$} }.
\end{eqnarray*}
Therefore, by definition that
\[ d_{X / G}([x],[y]) \leq d_{X/G}([x],[y]) + d_{X/G}([y],[z]). \]
Thus, $X / G$ is a metric space.
\end{proof}

When given a finitely generated group $G$ with a finite generated subgroup $H \leq G$, we note that $H$ acts translation-like on $G$ in a natural way by left multiplication; moreover, we have that the translation-like geometric of $G$ by $H$ is bi-Lipschitz to the coset space of $H$ in $G$. In general, a translation-like geometric quotient of a finitely generated group $G$ by a finitely generated group $H$ will not necessarily be bi-Lipschitz to the coset space of a subgroup $K \leq G$. Therefore, we may view the translation-like geometric quotient of $G$ by a finitely generated group $H$ is a generalization of coset spaces of subgroups.

The next propositions show that if given a UDBG space $X$ with a translation-like action by a group $G$, then the translation-like action geometric quotient is well-defined up to the bi-Lipschitz classes of $G$ and $X$.

\begin{prop}\label{bilip_equiv_coarse_model_1}
Let $X$ and $Y$ are UDBG spaces with a bi-Lipschitz equivalence $F\colon X \to Y$, and suppose that $G$ is a finitely generated group that acts translation-like on $X$. If we equip $Y$ with the translation-like action of $G$ induced by the bi-Lipschitz equivalence, then $X/G$ is bi-Lipschitz to $Y/G$.
\end{prop}

\begin{proof}
Let $d_X$ and $d_Y$ be the metrics of $X$ and $Y$, respectively. We claim that $F$ descends to a bijection between $X / G$ and $Y/G$. By Lemma \ref{bilipschitz_equivalent}, we have that the action of $G$ on $Y$ is given by $g \cdot y = F( g \cdot F^{-1}(y))$. If $g \cdot x_1 = x_2$ for $x_1, x_2 \in X$, we have that
\[ g \cdot F(x_1) = F( g \cdot F^{-1}(F(x_1))) = F( g \cdot x_1) = F(x_2). \]
Thus, the map $F$ preserves equivalence classes, and since the induced map $\bar{F}: X / G \to Y/G$ is clearly a bijection, we have our claim.

There exists a constant $C \geq 1$ such that for all elements $x,y \in X$, we have that 
\[ \frac{1}{C} d_X(x,y) \leq d_Y(F(x), F(y)) \leq C d_X(x,y). \] 
If $(p_1,q_1), \cdots, (p_n,q_n)$ is a chain from $x$ to $y$ in $X$, then $(F(p_1), F(q_1)), \cdots, (F(p_n), F(q_n))$ is a chain from $F(x)$ to $F(y)$. In particular, we have that
\[ d_{Y/ G}([\bar{F}(x)], [\bar{F}(y)]) \leq \sum_{i=1}^n d_Y(F(x), F(y)) \leq C \sum_{i=1}^n d_X(x,y). \]
By taking the infimum over all $n$-chains from $x$ to $y$, we have that 
\[ d_{Y/ G}([\bar{F}(x)], [\bar{F}(y)]) \leq C d_{X / G}([x],[y]). \]
Using similar arguments, we have that 
\[ \frac{1}{C} d_{X/ G}([x],[y]) \leq d_{Y/ G}([\bar{F}(x)], [\bar{F}(y)]) . \qedhere \] 
\end{proof}

\begin{prop}\label{bilip_equiv_coarse_model_2}
Let $X$ be a UDBG space, and suppose that $G$ is a finitely generated group that admits a translation-like action on $X$. If $H$ is bi-Lipschitz to $G$ via the map $F$, then with the induced translation-like action of $H$ on $X$, we have that $X/G$ is bi-Lipschitz to $X/H$.
\end{prop}

\begin{proof}
For simplicity in this proof, we go with the right action. Letting $\La$ be a set of orbit representatives of the action of $G$ on $H$, we have that $X = \bigsqcup_{x \in \La} x \cdot G$. We have that $H$ acts on itself via right multiplication, and thus, the action of $H$ on $X$ is given by 
\[ h \cdot (x \cdot g) = x \cdot ( F(F^{-1}(g) \cdot h^{-1})). \] 
We claim that $y_1 \sim y_2$ via the $G$-action if and only if $y_1 \sim y_2$ via the $H$--action. Suppose that $x$ represents the equivalence class of $y_1$ and $y_2$. There exist elements $g_1, g_2 \in G$ such that $x \cdot g_1 = y_1$ and $x \cdot  g_2 = y_2$. Since $H$ acts transitively on $G$, there exists an element $h \in H$ such that $g_1 \cdot h = g_2$. Therefore, $y_1 \cdot h = y_2$. The other direction is similar. As a consequence, we have that $(p_1, q_1), \cdots, (p_n, q_n)$ is a chain from $x$ to $y$ with respect to the $G$-action if and only if it is a chain from $x$ to $y$ with respect to the $H$--action. In particular
\[ d_{X/G}([x]_G,[y]_G) = d_{X/H}([x]_H, [y]_H). \]
By the above arguments, we have that the identity map from $X$ to itself  descends to a map of the orbit spaces $F\colon X /G \to X/H$ which is a bi-Lipschitz equivalence.
\end{proof}

\begin{defn}\label{D:CorMod}
Let $\textbf{G}$ be a Lie group with a Lie subgroup $\textbf{H} \leq \textbf{G}$. Let $\Ga < \textbf{G}$ and $\Delta < \textbf{H}$ be cocompact lattices. We say that a translation-like action of $\Delta$ on $\Gamma$ gives rise to a \textbf{coarse model of the homogeneous space $\textbf{G}/ \textbf{H}$} if there exists a UDBG space $X \subset \textbf{G}$ that bi-Lipschitz to $\Ga$ such that the orbits of the induced translation-like action of $\Delta$ on $X$ are coarsely embedded and contained in cosets of $\textbf{H}$ in $\textbf{G}$ and where there exists a natural bi-Lipschitz embedding from $X/\Delta$ to $\textbf{G} / \textbf{H}$ that is a quasi-isometry. 
\end{defn}

\subsection{Carnot Lie groups}

We are interested in a special class of nilpotent Lie algebras that admit natural dilations which act as a generalized notion of scaling.

\begin{defn}
Let $\Fr{g}$ be a nilpotent Lie algebra of step length $c$. We say that $\Fr{n}$ is a \textbf{stratified} nilpotent Lie algebra if it admits a grading $\Fr{n} = \bigoplus_{i=1}^c \Fr{v}_i$ where $\Fr{v}_1$ generates $\Fr{n}$. We say that a nilpotent Lie group $\textbf{N}$ is \textbf{stratified} if its Lie algebra is stratified. 
\end{defn}

Let $\Fr{n}$ be a stratified nilpotent Lie algebra of step size $c$ with grading $\bigoplus_{i=1}^c \Fr{v}_i$. Observe that the linear maps $\map{d \delta_t}{\Fr{n}}{\Fr{n}}$ given by
\[ d \delta_tX_1,\cdots,X_c) = \pr{t \cdot X_1, t^2 \cdot X_2, \cdots, t^c \cdot X_c} \]
satisfy $d \delta_t([X,Y]) = [d \delta_t(X), d \delta_t(Y)]$ and $d \delta_{ts} = d \delta_t \circ d \delta_s$ for $X,Y \in \Fr{g}$ and $t,s > 0$. Thus, $\set{d \delta_t ~:~ t > 0}$ gives a one parameter family of Lie automorphisms of $\Fr{n}$. If $\textbf{N}$ is a connected, simply connected nilpotent Lie group with Lie algebra $\Fr{n}$, then by exponentiating $d\delta_t$ we have an one parameter family of automorphisms denoted $\delta_t$. The \textbf{dilation on $\textbf{N}$ of factor $t$} is the Lie automorphism $\delta_t$.

We have the following lemma whose proof is an exercise in basic differential topology.
 
 \begin{lemma}\label{dilation_derivative}
Let $\textbf{N}$ be a connected, simply connected stratified nilpotent Lie group with Lie algebra $\Fr{n}$. Let $X \in \Fr{n}$, $t> 0$, and $x \in \textbf{N}$. If $V = L_{x}(X)$, then $(d \delta_t)_x (V)= (dL_x \circ d \delta_t)_1(X)$.
 \end{lemma}
 
\begin{proof} 
Since $\textbf{N}$ is a connected, simply connected nilpotent Lie group, the exponential map $\exp$ is a diffeomorphism whose inverse we formally denote as $\Log$. Letting $U$ be a small neighborhood about the identity, we have that $\pr{U, \Log}$ is a local chart around the identity.  Thus, we have that $(L_x(U), \varphi_x)$ is a local chart about $x$ where $\varphi_x = \Log \circ L_{x^{-1}}$. We then have that the map given by $\map{\varphi_x^{-1} \circ (d\delta_t)_1 \circ \varphi_x}{L_x(U)}{\delta_t(L_x(U))}$ is a local coordinate representation of $\delta_t$ at $x$. Thus, 
 \[ (d \delta_t)_x = (d\varphi_x)^{-1} \circ (d \delta_t)_1 \circ (d \varphi_x) = (d (L_x \circ \exp)) \circ (d \delta_t)_1 \circ d (\Log \circ  L_{x^{-1}}). \] 
Observing that $\textbf{N} \subset \GL(n,\R)$ and $\Fr{n} \subset \Fr{gl}(n,\R)$ for some $n$, we may write
 \[ (d \delta_t)_x(V) = x \: (d \exp)_1 \circ (d \delta_t)_1 \circ (d \Log)_1 (x^{-1} \:V). \]
There exist vectors $X_i \in \Fr{v}_i$ such that $V = \sum_{i=1}^c x \: X_i$. Since $\delta_t \circ \exp = \exp \circ \delta_t$, we have that $\Log \circ \delta_t = \delta_t \circ \Log$. In particular, we may write $(d \delta_t)_1 \circ (d \Log)_1 = (d \Log)_1 \circ (d \delta_t)_1$. Thus,
\[ (d \delta_t)_x(V)=(d \delta_t)_x\pr{x \: X}=x (d \exp)_1 \circ (d\delta_t)_1 \circ (d \Log)_1   \circ L_{x^{-1}}(x \: X)  = \pr{\sum_{i=1}^c x\: (d \delta_t)_1  \:X_i }. \]
Hence, 
\[ (d \delta_t)_x(V)=x \pr{\sum_{i=1}^c t^i \: X_i }=\sum_{i=1}^c (dL_x)_1(t^{i} \:X_i)=(dL_{x})_1 \pr{\sum_{i=1}^c t^i \: X_i} = (d L_{x})_1 \circ (d\delta_t)_1(X). \]
Therefore, $(d \delta_t)_x(V) = (dL_x \circ  \delta_t)_1(X)$.
 \end{proof}
 
\subsection{Semisimple Lie groups}

We recall standard facts in the theory of semisimple Lie groups which can be found in \cite{Eberlein,Helgason,Knapp,wolf_spaces_constant_curvature}. 

\begin{defn}
Given a real Lie algebra $\Fr{g}$, the \textbf{Killing form} is the symmetric bilinear form $\map{B}{\Fr{g} \times \Fr{g}}{\R}$ given by
\[ B_{\Fr{g}}(X,Y) = \mathrm{Tr}(\ad_X \circ \ad_Y). \]
We write $B = B_{\Fr{g}}$ when $\Fr{g}$ is clear from context. If $B$ is non-degenerate, we say that $\Fr{g}$ is a \textbf{semisimple Lie algebra}. If the Lie algebra of the Lie group $\textbf{G}$ is semisimple, we say that $\textbf{G}$ is a \textbf{semisimple Lie group}.
\end{defn}

\subsubsection{Iwasawa decomposition of a semisimple Lie group}

The Iwasawa decomposition of a semisimple Lie group $\textbf{G}$ arises from considerations of an involutive automorphism of the Lie algebra $\Fr{g}$.

\begin{defn}
An involution $\map{\theta}{\Fr{g}}{\Fr{g}}$ is called a \textbf{Cartan involution} if the bilinear form given by $B_\theta(X,Y) = -B(X,\theta(Y))$ is positive definite. We call the bilinear form $B_\theta$ the \textbf{Cartan-Killing metric} on $\textbf{G}$. Every real semisimple Lie algebra admits a Cartan involution, and any two Cartan involutions of a real semisimple Lie algebra differ by an inner automorphism.
\end{defn}

If $\theta$ is a Cartan involution of the semisimple Lie algebra $\Fr{g}$, then the Cartan decomposition is given by the vector space direct sum $\Fr{g} = \Fr{k} + \Fr{p}$ where $\Fr{k}$ and $\Fr{p}$ are the eigenspaces relative to the eigenvalues $1$ and $-1$ of $\theta$. We fix a maximal abelian subspace $\Fr{a}$ of $\Fr{p}$, with $\dim \mathfrak{a} = \text{rank}(\textbf{G})$. The Cartan decomposition is orthogonal with respect to the bilinear form $B_\theta(X,Y)$. We fix an order on the system $R \subseteq \Fr{a}^\prime$ of non-zero restricted roots of $\pr{\Fr{g},\Fr{a}}$. Let 
\[ \Fr{m} = \set{X \in \Fr{k} ~:~ [X,Y] = 0 \text{ for all } Y \in \Fr{a}}. \] 
The Lie algebra $\Fr{g}$ decomposes as
\[ \Fr{g} = \Fr{m} + \Fr{a} + \bigoplus_{\al \in R}\Fr{g}_\al \] 
where $\Fr{g}_\al$ is the root space relative to the root $\al$. We denote $\Pi^+$ as the subset of positive roots. If $\textbf{K}$, $\textbf{A}$, and $\textbf{N}$ are the Lie subgroups with Lie algebras $\Fr{k}$, $\Fr{a}$ and $\Fr{n} = \oplus_{\al \in \Pi_+}\Fr{g}_\al$, then the map from $\textbf{K}\times \textbf{A} \times \textbf{N}$ to $\textbf{G}$ given by $\pr{k,a,n} \rightarrow kan$ is a diffeomorphism. In particular, we write $\textbf{G} = \textbf{K} \textbf{A} \textbf{N}$ and call this the \textbf{Iwasawa decomposition of $\textbf{G}$.} We have that $\textbf{K}$ is a compact Lie group, $\textbf{A}$ is a connected, simply connected abelian Lie group, and $\textbf{N}$ is a connected, simply connected nilpotent Lie group. Moreover, we have that $\textbf{N}$ has additional structure in that $\textbf{N}$ is a stratified nilpotent group as shown below.

Denote by $\Phi$ the subset of positive simple roots. Given that root spaces satisfy $[\Fr{g}_\al,\Fr{g}_\beta] \subseteq \Fr{g}_{\al + \beta}$, the subspace $V \subseteq \Fr{n}$ given by $V = \bigoplus_{\delta \in \Phi} \Fr{g}_\delta$ provides a stratification of $\Fr{n}$. In particular, $\Fr{n}$ is a stratified nilpotent Lie algebra and thus, $\textbf{N}$ is a stratified nilpotent Lie group. 
We write this down as a proposition.

\begin{prop}
Let $\textbf{G}$ be a connected, semisimple Lie group, and let $\textbf{G} = \textbf{KAN}$ be an Iwasawa decomposition. Then $\textbf{N}$ is a stratified nilpotent Lie group.
\end{prop}

We introduce some notation. Assuming that $\textbf{N}$ has step length $c$, we denote $\Phi_{i}$ as the set of roots such that $\Fr{n}_i / \Fr{n}_{i+1} = \bigoplus_{\beta \in \Pi_i} \Fr{g}_{\beta}$ as vector spaces with some ordering on the roots. Since $\textbf{N}$ is a connected, simply connected nilpotent Lie group, the exponential map is a diffeomorphism. In particular, the Baker-Campbell-Hausdorff formula implies that $\textbf{N}$ is diffeomorphic to $\prod_{i=1}^{c} \prod_{\beta \in \Pi_i} \exp(\Fr{g}_{\beta})$.

\section{Metrics on semisimple Lie groups}

For semisimple Lie groups $\textbf{G}$ with maximal compact subgroup $\textbf{K}$, we  have that $\textbf{G} / \textbf{K} = \R^{\text{rank}(\textbf{G})} \times \textbf{N}$ as smooth spaces. If $g$ is the Cartan-Killing metric on $\textbf{G}$, then at the identity coset of $\textbf{G} / \textbf{K}$, we have by \cite[Section 4]{Borel} for $\pr{a,n} \in \textbf{G} / \textbf{K}$ that
\[ g_{a,n} = \sum_{i=1}^{\text{rank}(\textbf{G})}da_i^2 + \sum_{i=1}^{c(\textbf{N})}\sum_{\beta \in \Pi_i} \beta(a) (g_{\beta})_n \]
where $\sum_{\beta \in \Phi} g_{\beta}$ is a left-invariant metric on $\Fr{n}$, the Lie algebra of $\textbf{N}$. If $c\colon [0,1] \to \textbf{G} / \textbf{K}$ is a smooth curve, we may write 
\[ c(t) = \pr{c_{a}(t), \pr{c_{\beta,1}(t)}_{\beta \in \Phi_1}, \cdots, \pr{c_{\beta,c(\textbf{N})}(t)}_{\beta \in \Phi_{c(\textbf{N})}}} \]
where $c_a\colon [0,1] \to \mathbb{R}^{\text{rank}(\textbf{G})}$ is a smooth math and $c_{\beta,i}\colon [0,1] \to \exp(\Fr{g}_\beta)$ is a smooth map for all $\beta \in \Pi_i$ and $1 \leq i \leq c(\textbf{N})$. Thus, it is evident that $\R^{\text{rank}(\textbf{G})}$ with the standard flat metric, which we denote as $\abs{\cdot }$, is isometrically embedded. Since any vector $X \in \Fr{n}$ may be written as 
\[ X = \sum_{i=1}^{c} \sum_{\beta \in \Pi_i} X_{\beta} \] 
where $X_ \beta \in \Fr{g}_\beta$, we may write the length of $c$ with respect to the metric $g_{a,n}$ as
\[ \ell_{\textbf{G}/\textbf{K}}(c) = \mathlarger{\mathlarger{\int}}_{0}^1 \sqrt{\sum_{j=1}^{\text{rank}(\textbf{G})} (dc_{a_j}(t))^2 + \sum_{t=1}^{c(\textbf{N})} \sum_{\beta \in \Pi_i} \beta(a) g_{\beta}(dc_{\beta}(t),dc_{\beta}(t))}dt. \]
The associated distance function on $\textbf{G} / \textbf{K}$ is given by
\[ d_{\textbf{G} / \textbf{K}} = \inf \set{ \ell_{\textbf{G} / \textbf{K}}(c) ~:~ c \text{ is a smooth path in } \textbf{G} / \textbf{K} \text{ from } x \text{ to } y }. \]
For $a \in \mathbb{R}^{\text{rank}(\textbf{G})}$, we denote $\textbf{N}_a$ as the nilpotent Lie group $\textbf{N}$ equipped with the left invariant metric
\[ \sum_{i=1}^{c(\textbf{N})}\sum_{\beta \in \Pi_i} \beta(a) (g_{\beta})_n \]
which we will identify with $\set{a} \times \textbf{N}$ in $\textbf{G} / \textbf{K}$. Any smooth curve $c\colon [0,1] \to \textbf{N}_a$ has the form 
\[ c(t) =\pr{ \pr{c_{\beta,1}(t)}_{\beta \in \Phi_1}, \cdots, \pr{c_{\beta,c(\textbf{N})}(t)}_{\beta \in \Phi_{c(\textbf{N})}}} \] 
where $c_{\beta,i}(t) \in \exp(\Fr{g}_{\beta})$ for all $t \in [0,1]$. Therefore, the length of $c$ in $\textbf{N}_a$ is given by
\[ \ell_a(c) = \mathlarger{\mathlarger{\int}}_{0}^1  \sqrt{\sum_{i=1}^{c(\textbf{N}_a)} \sum_{\beta \in \Pi_i} \beta(a) g_{\beta}(dc_{\beta,i}(t),dc_{\beta,i}(t))}dt. \]
As before, the associated distance function is given by
\[ d_a(x,y) = \inf \set{ \ell_a(c) ~:~ c \text{ is a smooth path in } \textbf{N}_a \text{ from } x \text{ to } y }. \]
We have the following smooth diffeomorphism which dilates $\textbf{N}$ based on the point $a$ in $\mathbb{R}^{\text{rank}(\textbf{G})}.$

\begin{defn}
For each $1 \leq i \leq c(\textbf{N})$ and $\beta \in \Pi_i$, we denote $f_{\beta,i}(a)= 1 / \sqrt[2i]{\beta(a)}$. With this value, we denote the following map $F_a\colon \textbf{N} \to \textbf{N}$ as
\[ F_{a}(x) = \pr{\delta_{f_{\beta,i}(a)}(x_{\beta,i})}_{1 \leq i \leq c(\textbf{N}), \beta \in \Pi_i}. \]
Since $\delta_{\beta,i}$ is a smooth map for all $\beta \in \Pi_i$ and each $1 \leq i \leq c(\textbf{N})$, we have that $F$ is a diffeomorphism.
\end{defn}

We note for all elements $a \in \mathbb{R}^{\text{rank}(\textbf{G})}$ and roots $\beta \in \Phi$ that $\beta(a) > 0$. In particular, we have that $\beta(\vec{0}) = 1$ for all $\beta \in \Phi$. With this observation in mind, we have the following proposition which relates the length of the path $c$ in $\textbf{N}_{\vec{0}}$ to length of the path in $F_a(c)$ in $\textbf{N}_a$.

\begin{prop}
If $c\colon[0,1] \to \textbf{N}$ is a smooth curve, then for all $a \in \mathbb{R}^{\text{rank}(\textbf{G})}$ we have that $\ell_a(F_a(c)) = \ell_{0}(c)$.
\end{prop}

\begin{proof}
We have that 
\begin{align*}
c(t) &= ( \pr{c_{\beta,1}(t)}_{\beta \in \Phi_1}, \cdots, \pr{c_{\beta,c(\textbf{N})}(t)}_{\beta \in \Phi_{c(\textbf{N})}}) \\
dc(t) &=  (\pr{dc_{\beta,1}(t)}_{\beta \in \Phi_1}, \cdots, \pr{dc_{\beta,c(\textbf{N})}(t)}_{\beta \in \Phi_{c(\textbf{N})}}).
\end{align*}
We may write 
\[ dc_{\beta,i}(t) = dL_{c_{\beta,i}(t)}(X_{\beta,i}(t)) \] 
where $X_{\beta,i}\colon [0,1] \to \Fr{g}_{\beta}$ is a smooth function. For notational simplicity, we let $\rho_{\beta,i,a}(t) = \delta_{f_{\beta,i}(a)} \circ c_{\beta,i}(t)$. Thus, Lemma \ref{dilation_derivative} implies that
\[ d(\delta_{f_{\beta,i}(a)}\circ c_{\beta,i})(t) = (\delta_{f_{\beta,i}(a)})_{1}(dL_{\rho_{\beta,i,a}(t)}(X_{\beta,i}(t))) = (1/ \sqrt[2]{\beta(a)}) dL_{c_{\beta,i}(t)}(X_{\beta,i}(t)). \]
Therefore, we may write
\begin{eqnarray*}
\beta(a)g_{\beta}(d(\rho_{\beta,i,a}(t)),d(\rho_{\beta,i,a}(t))_{\rho_{\beta,i,a}(t)} & = & g_{\beta}( dL_{c_{\beta,i}(t)}(X_{\beta,i}(t)), dL_{c_{\beta,i}(t)}(X_{\beta,i}(t)) )_{\rho_{\beta,i,a}(t)}  \\
 &=& g_{\beta}(X_\beta(t),X_\beta(t))_1\\
&=& g_{\beta}(dL_{c_{\beta,i}}(X(t)),dL_{c_{\beta,i}}(t))_1\\
&=& g_\beta(dc_{\beta,i}(t),dc_{\beta,i}(t)).
\end{eqnarray*}
Combining everything together, we may write
\begin{eqnarray*}
\ell_a(F_a(c)) &=& \mathlarger{\mathlarger{\int}}_{0}^1 \sqrt{\sum_{i=1}^{c(\textbf{N}_a)} \sum_{\beta \in \Pi_i} \beta(a) g_{\beta}(d\rho_{\beta,i,a}(t), d\rho_{\beta,i,a}(t))_{\rho_{\beta,i,a}(t)}}dt \\ &=&\mathlarger{\mathlarger{\int}}_{0}^1 \sqrt{ \sum_{i=1}^{
c(\textbf{N}_a)} \sum_{\beta \in \Pi_i} \beta(a) g_{\beta}(dc_{\beta,i}(t), dc_{\beta,i}(t))_{c_{\beta,i}(t)}}dt\\ 
&=& \ell_{0}(c). \qedhere
\end{eqnarray*}
\end{proof}

As a natural consequence, we have the following corollary.

\begin{cor}\label{cor_dilation_distance}
Let $x,y \in \textbf{N}$, and let $a \in \mathbb{R}^{\text{rank}(\textbf{G})}$. Then $d_a(F_a(x),F_a(y)) = d_{0}(x,y)$.
\end{cor}

\begin{proof}
Let $c$ be a smooth path from $x$ to $y$. We have by the above proposition that $\ell_{0}(c) = \ell_a(F_a\circ c)$. Since $F_a \circ c$ is a path from $F_a(x)$ to $F_a(y)$, we have that 
\[ d_a(F_a(x),F_a(y)) \le \ell_a(F_a \circ c) = \ell_{0}(c). \] 
Therefore, by definition, we have that $d_{\vec{0}}(x,y) \leq d_a(F_a(x),F_a(y))$. Using a similar argument, we also have that $d_a(F_a(x),F_a(y)) \leq d_{0}(x,y)$. Therefore, $d_{0}(x,y) = d_a(F_a(x),F_a(y))$.
\end{proof}

We now provide a lower bound for the distance between points in distinct cosets of $\textbf{N}$ in terms of the distance between of the the coordinates of the coset representatives.

\begin{lemma}\label{distance_between_flats}
Let $x,y$ be distinct points in $\R^{\text{rank}(\textbf{G})}$, and let $g,h \in \textbf{N}$.  We then have that $d_{\textbf{G} / \textbf{K}}((x,g),(y,h)) \geq \abs{x - y}.$ Moreover, if $g =h$, then $d_{\textbf{G} / \textbf{K}}((x,g),(y,g)) = \abs{x- y}$.
\end{lemma}

\begin{proof}
Let $c$ be a path between $(x,g)$ and $(y,h)$. We may write
\begin{eqnarray*}
\ell_{\textbf{G} / \textbf{K}}(c) &=& \mathlarger{\mathlarger{\int}}_{0}^1 \sqrt{\sum_{j=1}^{\text{rank}(\textbf{G})} (dc_{a_j}(t))^2 + \sum_{i=1}^{c(\textbf{N})} \sum_{\beta \in \Pi_i} \beta(a) g_{\beta}(dc_{\beta,i}(t),dc_{\beta,i}(t))}dt\\
&\geq& \mathlarger{\mathlarger{\int}}_{0}^1 \sqrt{\sum_{j=1}^{c(\textbf{N})} (dc_{a_{\beta,i}}(t))^2}dt =\int_{0}^1 \abs{dc_{a}(t)} \geq \abs{x - y}.
\end{eqnarray*}
Therefore, we have by definition that $d_{\textbf{G} / \textbf{K}}((x,g),(y,h)) \geq \abs{x-y}$.

Let $\ga\colon [0,1] \to \R^{\text{rank}(\textbf{G})}$ be a straight line path from $x$ to $y$, and let $c\colon [0,1] \to \textbf{G} / \textbf{K}$ be the path given by $c(t) = (\ga(t),g)$. We may express the length of $c$ as
\[ \ell_{\textbf{G} / \textbf{K}}(c) = \mathlarger{\mathlarger{\int}}_{0}^1 \sqrt{\sum_{j=1}^{\text{rank}(\textbf{G})}(dc_{a_i}(t))^2}dt = \int_0^1 \abs{d\ga(t)}dt = \abs{x - y}. \]
In particular, we have that $d_{\textbf{G} / \textbf{K}}((x,g),(y,g)) \leq \abs{x-y}$. Using the above inequality, we have that 
\[ d_{\textbf{G} / \textbf{K}}((x,g),(y,g)) = \abs{x-y}. \qedhere \]
\end{proof}

The last proposition of this section relates the distance between $(\vec{0},x)$ and $(\vec{0},y)$  in $\textbf{N}_{\vec{0}}$ with the distance between $(a,F_a(x))$ and $(a,F_a(y))$ for any $a \in \R^{\rank(\textbf{G})}$ as points in $\textbf{G} / \textbf{K}$.
\begin{prop}\label{iwasawa_distortion}
Let $g,h \in \textbf{N}$, and let $a \in \R^{\rank(\textbf{G})}$. Then
\[ C_1 \ln(d_{0}(x,y)) \leq d_{\textbf{G} / \textbf{K}}((a,F_a(x)),(a,F_a(y))) \leq C_2 \ln(d_{0}(x,y)) \]
for some constants $C_1, C_2 > 0$.
\end{prop}

\begin{proof}
By \cite[$3.C_1'$]{gromov}, we have that there exist constants $C_1, C_2 > 0$ such that
\[ C_1 \: \ln( d_a(F_a(x),F_a(y))) \leq d_{\textbf{G} / \textbf{K}}((a,F_a(x)),(a,F_a(y))) \leq C_2 \: \ln(d_a(F_a(x),F_a(y))). \]
By Corollary \ref{cor_dilation_distance}, we have that $d_a(F_a(x),F_a(y)) = d_0(x,y)$. Thus, we have that
\[ C_1 \: \ln(d_0(x,y)) \leq d_{\textbf{G} / \textbf{K}}((a,F_a(x)),(a,F_a(y))) \leq C_2 \: \ln(d_0(x,y)). \qedhere \]
\end{proof}

\section{Lipchitz models for cocompact lattices in semisimple Lie groups}

We now introduce a model for the Lipschitz geometry of cocompact lattices in an arbitrary semisimple Lie group $\textbf{G}$ with an Iwasawa decomposition $\textbf{G} = \textbf{KAN}$. For a cocompact lattice $\Delta \subset \textbf{N}$, we let $\textbf{X}(\Delta) \subset \textbf{G} / \textbf{K}$ be the subset given by
\[ \textbf{X}(\Delta) = \set{ (a, F_a(g)) ~:~ a \in \mathbb{Z}^{\rank(\textbf{G})}, g \in \Delta } \]
with the induced metric.

\begin{prop}\label{main_prop}
If $\Delta < \textbf{N}$ be a cocompact lattice such that
\[ \inf \{ d_0(x,y) ~:~ x,y \in \Delta, x \neq y \} > 1, \] 
then $\textbf{X}(\Delta)$ is a $UDBG$ space.
\end{prop}

\begin{proof}
We first show that $\textbf{X}(\Delta)$ is uniformly discrete. If $x, y \in \mathbb{R}^{\rank(\textbf{G})}$ such that $x \neq y$, then Lemma \ref{distance_between_flats} implies for any $g,h \in \Delta$ that
\[ d_{\textbf{G} / \textbf{K}}((x,F_x(g)),(y,F_y(h))) \geq \abs{x-y} \geq 1. \]

For $z =x = y$, Proposition \ref{iwasawa_distortion} implies that there exists a constant $C_1 > 0$ such that
\[ d_{\textbf{G} / \textbf{K}}((z,F_z(g)), (z,F_z(h))) \geq C_1 \: \ln(d_0(g,h)) \geq C_1 \: \ln (\inf \set{d_0(a,b) ~:~ a,b \in \Delta, a \neq b}). \]
Therefore, for all $(x, F_x(g)), (y, F_y(h)) \in \textbf{X}(\Delta)$, we have that
\[ d_{\textbf{G} / \textbf{K}}((x,F_x(g)),(y,F_y(h))) \geq \min\set{1, C_1 \: \ln (\inf \set{d_0(a,b) ~:~ a,b \in \Delta, a \neq b })}. \]
In particular, we have that
\[ \inf\set{ d_{\textbf{G} / \textbf{K}}((x,F_x(g)),(y,F_y(h))) ~:~ (x,F_x(g)) \neq (y,F_y(h)) \text{ in } \textbf{X}(\Delta)  } > 0 \]
showing that $\textbf{X}(\Delta)$ is uniformly discrete.

We now demonstrate that $\textbf{X}(\Delta)$ has bounded geometry. To do that, we show for all $r > 0$ that there exists a constant $C_r$ such that $\abs{B_{\textbf{X}(\Delta)}((x,F_x(g)))} \leq C_r$ for all $g \in \Delta$ and $x \in \Z^{\rank(\textbf{G})}$. We start by showing that there exists a universal constant $M_r$ such that any $r$--ball in $\textbf{X}(\Delta)$ intersects at most $M_r$ sets of the form $(x,\textbf{N})$ where $x \in \Z^{\rank(\textbf{G})}$. We also need to show that there exists a constant $C_r > 0$ such that 
\[ \abs{B_{\textbf{X}(\Delta),r}(x,F_x(g)) \cap (y,\textbf{N})} \leq C_r \] 
for $y \in \Z^{\rank(\textbf{G})}$.

If $(y,F_y(h)) \in B_{\textbf{X}(\Delta),r}((x,F_x(g)))$ such that $\abs{x - y} > r$,  then Proposition \ref{distance_between_flats} implies that
\[ d_{\textbf{G} / \textbf{K}}((y,F_y(h)),(x,F_x(g))) \geq \abs{x -y} > r \]
which is a contradiction. Therefore, we have that $\abs{x-y} \leq r$. Since $x$ is fixed, it is easy to see that exists a constant $M > 0$ such that there are at most $M$ sets of the form $(y,\textbf{N})$ such that 
\[ B_{\textbf{X}(\Delta),r}((x,F_x(g))) \cap (y, \textbf{N}) \neq \emptyset. \]

First consider $(x, F_x(h)) \in B_{\textbf{X}(\Delta),r}((x,F_x(g)))$. We have by the above reasoning that $\abs{x-y} \leq r$, and thus, the triangle inequality and Corollary \ref{cor_dilation_distance} imply that 
\[ d_0(g,h) \leq C_1 e^{d_{\textbf{G} / \textbf{K}}((x,F_x(h)),(x,F_x(g)))} \leq C_1 e^r. \]
Thus, $h \in B_{\Delta, C_1 e^r}(g)$, and by Gromov's polynomial growth theorem, we have that there exists a constant $C_2 > 0$ and a natural number $d$ such that $\abs{B_{\Delta, C_1 e^r}(g)} \leq C_2 \: C_1^d \: e^{dr}.$ Now consider $(y,F_y(h)) \in B_{\textbf{X}(\Delta),r}((x,F_x(g)))$ where $x \neq y$. That implies 
\[ (y,\textbf{N}) \cap B_{\textbf{X}(\Delta),r}((x,F_x(g)) \neq \emptyset, \] 
and by the above statement, we have that there exist at most $M_r$ such points $y$. Taking these statements together, we have that 
\[ \abs{B_{\textbf{X}(\Delta),r}((x,F_x(g))} \leq C_3 e^{dr} \] 
for some constant $C_3 > 0$. Therefore, $\textbf{X}(\Delta)$ is a UDBG space.
\end{proof}

The following proposition demonstrates that under appropriate assumptions on $\Delta < \textbf{N}$ that $\textbf{X}(\Delta)$ can be thought of as a model for the Lipschitz geometry of $\Ga$ where $\Ga < \textbf{G}$ is a cocompact lattice.

\begin{prop}\label{bilipschitz_model}
Let $\Ga < \textbf{G}$ be a cocompact lattice, and let $\Delta < \textbf{N}$ be a cocompact lattice satisfying
\[ \inf\set{d_0(g,h) ~:~ g,h \in \Delta, g \neq h } > 1. \]
Then $\textbf{X}(\Delta)$ is bi-Lipschitz to $\Ga$.
\end{prop}

\begin{proof}
Let $C_1, C_2 > 0$ be the constants from Proposition \ref{iwasawa_distortion}, and let $C = \max\set{C_1,C_2}$.  Since $\Delta$ is a cocompact lattice in $\textbf{N}$, we have that $\Delta$ is quasi-isometric to $\textbf{N}$. In particular, there exists a constant $\eps_1 > 0$ such that if $g \in \textbf{N}$, then there exists an element $h \in \Delta$ such that $d_0(g,h) \leq \eps_1$. Letting $\eps = C 
\: \eps_1 +\rank(\textbf{G})$, we claim that $\textbf{X}(\Delta)$ is $\eps$--dense in $\textbf{G} / \textbf{K}$. Let $(x,g) \in \textbf{X}$ where $x \in \R^{\rank(\textbf{G})}$ and $g \in \textbf{N}$. 

Suppose that $x \in \Z^{\rank(\textbf{G})}$. There exists an element $h \in \Delta$ such that $d_1(F_{-x}(g),h) \leq \eps_1$.  Proposition \ref{iwasawa_distortion} implies that
\[ d_{\textbf{G} / \textbf{K}}((x,F_x(h)),(x,g)) \leq d_0(h, (F_{-x}(g)) \leq C \: \eps_1. \]

Suppose that $x \in \R^{\rank(\textbf{G})}  \backslash \Z^{\rank(\textbf{G})}$. There exists an element $y \in \Z^{\rank(\textbf{G})}$ such that $\abs{x-y} \leq 2\rank(\textbf{G})$, and thus, there exists an element $h \in \Delta$ such that $d_0(F_{-y}(g),h) \leq \eps_1$. By the triangle inequality, Lemma \ref{distance_between_flats}, and Proposition \ref{iwasawa_distortion}, we have that
\begin{eqnarray*}
d_{\textbf{G} / \textbf{K}}((x,g),(y,F_{y}(h))) &\leq& d_{\textbf{X}}((x,g),(y,g)) + d_{\textbf{X}}((y,g),(y,F_y(h)))\\
&\leq&  2\rank(\textbf{G}) + d_{\vec{0}}(F_{-y}(g),h)  \\
& \leq & 2 \rank(\textbf{G}) + \eps_1 = \eps.
\end{eqnarray*}
Therefore, $\textbf{X}(\Delta)$ is $\eps$-dense in $\textbf{G} / \textbf{K}$, and subsequently, $\textbf{X}(\Delta)$ and $\Gamma$ are quasi-isometric. Since $\textbf{X}(\Delta)$ and $\Gamma$ are quasi-isometric non-amenable spaces, Proposition \ref{udbg_nonamenable} implies that they are bi-lipschitz.
\end{proof}

\section{Proof of Theorem \ref{main_result}}

For the readers convenience, we restate our Theorem \ref{main_result}.

\begin{customthm}{\ref{main_result}}
Let $\B{G}$ be a semisimple Lie group with an Iwasawa decomposition $\B{G} = \B{K}\B{A}\B{N}$. If $\Gamma < \B{G}$ and $\Delta < \B{N}$ are cocompact lattices, then $\Gamma$ admits a translation-like action by $\Delta$. Moreover, we can choose this translation-like action to give rise to a coarse model $\Ga / \Delta$ of the homogeneous space $\B{G} / \B{N}$. Finally, given distinct lattices $\Ga_1, \Ga_2 < \B{G}$ and $\Delta_1, \Delta_2 < \B{N}$, we have the coarse models $\Ga_1/\Delta_1$ and $\Ga_2 / \Delta_2$ for $\B{G} / \B{N}$ are bi-Lipschitz.
\end{customthm}

\begin{proof}
It is evident that there exists a cocompact lattice $\Delta^\prime$ in $\B{N}$ satisfying
\[ \inf\set{ d_0(g,h) \: | \: g,h \in \Delta^\prime, g \neq h } > 1. \]
We first demonstrate that $\Delta^\prime$ admits a translation-like action on $\textbf{X}(\Delta^\prime)$. For $g \in \Delta^\prime$ and $(x,F_x(h)) \in \textbf{X}(\Delta^\prime)$, we let $g \cdot (x,F_x(h)) = (x,F_x(hg^{-1}))$. It is easy to see that this is a free action. Therefore, we need to demonstrate that we have a wobbling action. We have by Proposition \ref{iwasawa_distortion} that there exists a constant $C> 0$ such that
\[ d_{\textbf{G} / \textbf{K}}((x,F_x(h))(x,F_x(hg^{-1}))) \leq C \ln(d_0(h,hg^{-1})) \leq C \ln(d_0(1,g)). \]
Therefore, $\Delta^\prime$ admits a translation-like action on $\textbf{X}(\Delta^\prime)$.

To finish, we note that $\Delta$ is a cocompact lattice in $\Nilp$, and by the discussion after \cite[Ques 2]{burago_kleiner}, we have that $\Delta$ and $\Delta'$ are bi-Lipschitz. We have that $\Delta$ acts on itself by right multiplication, and thus, Lemma \ref{bilipschitz_equivalent} implies that $\Delta$ admits a translation-like action on $\Delta'$. Lemma \ref{bilipschitz_acting_groups} implies that $\Delta$ admits a translation-like action on $\textbf{X}(\Delta')$. Since $\textbf{X}(\Delta')$ is bi-Lipschitz to $\Gamma$, we have by Lemma \ref{bilipschitz_equivalent} that $\Delta$ admits a translation-like action on $\Gamma$ as desired. It is evident that the given translation-like action gives rise to a coarse model for $\textbf{G} / \textbf{N}$.

If $\Ga, \Ga^\prime < \B{G}$ are cocompact lattices, we have that $\Ga$ and $\Ga^\prime$ are bi-Lipschitz by Proposition \ref{udbg_nonamenable}. Moreover, if $\Delta, \Delta^\prime < \B{N}$ are cocompact lattices, then since $\B{N}$ is a Carnot group, we have by the remark after \cite[Ques 2]{burago_kleiner} that $\Delta$ and $\Delta^\prime$ are blipschitz. hus, by applying Proposition \ref{bilip_equiv_coarse_model_1} and Proposition \ref{bilip_equiv_coarse_model_2}, we see that $\Ga / \Delta$ and $\Ga^\prime / \Delta^\prime$ are bi-Lipschitz.
\end{proof}

For the proof of Corollary \ref{main_cor}, we note that if $\B{G}$ is not isogenous to $\SL(2,\R)$, then $\Z^2 \leq \Delta$. Since $\Z^2$ acts translation-like on $\Delta$ by virtue of being a subgroup, we have by Lemma \ref{transitivity_translation_like} that $\Z^2$ acts translation-like $\Gamma$.

\section{Proof of Theorem \ref{second_result}}

We restate Theorem \ref{second_result} for the reader's convenience.

\begin{customthm}{\ref{second_result}}
Let $\B{G}$ and $\B{H}$ be $\Q$--defined noncompact real simple Lie groups such that $\B{H} \leq \B{G}$. If $\Delta < \B{H}$ and $\Gamma < \B{G}$ are cocompact lattices, then $\Delta$ admits a translation-like action on $\Gamma$. Moreover, we can choose this translation-like such that $\Ga / \Delta$ is a coarse model for $\B{G} / \B{H}$. Finally, given distinct lattices $\Ga_1, \Ga_2 < \B{G}$ and $\Delta_1, \Delta_2 < \B{H}$, the spaces $\Ga_1/\Delta_1$ and $\Ga_2 / \Delta_2$ for $\B{G} / \B{H}$ are bi-Lipschitz.
\end{customthm}

\begin{proof}
Since the inclusion of $\B{H}$ into $\B{G}$ is  $\Q$--defined, we have by \cite[10.14. Corollary (iii)]{rag} that $\B{H}(\Z)$ is a subgroup of a cocompact lattice $\La$ that is commensurable with $\B{G}(\Z)$. We have that $\B{H}(\Z) \leq \B{H} \cap \La \leq \La$, and thus, $\B{H} \cap \La$ is a cocompact lattice in $\B{H}$. Hence, we have that $\La / \B{\Z} \cap \La$ naturally embeds into $\B{G} / \B{H}$ as a coarse dense subset. Thus, subgroup containment of $\B{H} \cap \La$ into $\La$ is a translation-like action that gives rise to a coarse model for the homogeneous space $\B{G} / \B{H}$. Since $\Ga$ and $\La$ are quasi-isometric non-amenable spaces, Proposition \ref{udbg_nonamenable} implies that $\Ga$ and $\La$ are bi-Lipschitz, and thus, $\B{H} \cap \La$ admits a translation-like action by Lemma \ref{bilipschitz_equivalent}. Hence, Proposition \ref{bilip_equiv_coarse_model_1} implies that $\Ga / \B{H} \cap \La$ is bi-Lipschitz to $\La / \B{H} \cap \La$, and thus, the translation-like action of $\B{H} \cap \La$ on $\Ga$ gives rise to a coarse model of $\B{G} / \B{H}$. Additionally, $\Delta$ and $\B{H} \cap \La$ are quasi-isometric nonamenable spaces, and thus, by Proposition \ref{udbg_nonamenable}, we have that they are bi-Lipschitz. Thus, Lemma \ref{bilipschitz_acting_groups} implies that $\Delta$ admits a natural translation-like action on $\Gamma$. Moreover, we have by Proposition \ref{bilip_equiv_coarse_model_2} that $\Gamma / \Delta$ is bi-Lipschitz to $\Ga / \B{H} \cap \La$. Subsequently, $\Delta$ admits a translation-like action on $\Gamma$ that gives rise to a coarse model for $\B{G} / \B{H}$. Finally, we note that if $\Gamma^\prime < \B{G}$ and $\Delta^\prime < \B{H}$ are different cocompact lattices, then by Proposition \ref{udbg_nonamenable}, we have that $\Delta$ and $\Delta^\prime$ are bi-Lipschitz and that $\Gamma$ and $\Gamma^\prime$ are bi-Lipschitz. Thus, by applying Proposition \ref{bilip_equiv_coarse_model_1} and Proposition \ref{bilip_equiv_coarse_model_2}, we see that $\Ga / \Delta$ and $\Ga^\prime / \Delta^\prime$ are bi-Lipschitz.
\end{proof}

\bibliographystyle{plain}


\end{document}